\documentclass[11pt]{amsart}

\usepackage[margin=2.7cm]{geometry}
\usepackage[parfill]{parskip}
\usepackage{graphicx}
\usepackage{float}
\usepackage{amsmath}
\usepackage{amssymb}
\usepackage{amsthm}
\usepackage{array}
\usepackage{mathrsfs}
\usepackage[all]{xy}
\usepackage{fancyhdr}
\usepackage{enumitem}
\usepackage{subfig}
\usepackage{tabularx}

\theoremstyle{plain}
\newtheorem{theorem}{Theorem}[section]

\newtheorem{lemma}[theorem]{Lemma}
\newtheorem{prop}[theorem]{Proposition}

\theoremstyle{definition}

\theoremstyle{remark}
\newtheorem{remark}[theorem]{Remark}

\DeclareMathOperator{\CP2}{\mathbb{C}P^2}
\DeclareMathOperator{\CPo}{\mathbb{C}P^1}

\DeclareMathOperator{\barCP2}{\overline{\mathbb{C}P^2}}

\DeclareMathOperator{\Z}{\mathbb{Z}}

\title{Comparing star surgery to rational blow-down}
\author{Laura Starkston}

\begin{document}
\maketitle

\begin{abstract}
	We compare the star surgery operations introduced in \cite{KS} to the generalized rational blow-down. We show that star surgery shares the properties that make rational blow-down useful for constructions of small exotic symplectic 4-manifolds. Then we show that star surgery operations provide a strictly more general class of operations by proving that there is an infinite family of star surgeries which are inequivalent to any sequence of generalized symplectic rational blow-downs. This answers a question posed to the author by Ozbagci. It also demonstrates that the monodromy substitutions coming from star surgery operations yield relations in planar mapping class monoids which cannot be positively generated by the relations determined in \cite{EMV} which come from the generalized rational blow-downs.
\end{abstract}

\section{Introduction}

	The rational blow-down operation on 4-manifolds was first defined by Fintushel and Stern \cite{FS}, generalized by Park \cite{P} and Stipsicz, Szab\'{o}, and Wahl \cite{SSW}, and shown to be symplectic by Symington \cite{Sy1, Sy2}. A rational blow-down cuts out a neighborhood of spheres intersecting in a particular way and glues in a rational homology ball along the common boundary. More general operations, called star surgery, were introduced by Karakurt and the author in \cite{KS}. These operations similarly cut out a convex neighborhood of spheres intersecting according to a star shaped graph, but the piece which is glued in can be any convex symplectic filling of the corresponding contact manifold, and need not be a rational homology ball. This allows for a much larger range of operations. The effect on the Seiberg-Witten invariants is understood and similar for the rational blow-down and the more general star surgery operations, and both can be used to construct symplectic manifolds with exotic smooth structures (see for example \cite{P, SS, M, KS}). The rational blow-down was particularly useful in constructing examples with small Euler characteristic because it kills off many generators of second homology. We prove here that this is generally true of all star surgery operations.
	
	\begin{theorem}\label{thm:smaller}
		Every nontrivial star surgery reduces the Euler characteristic and second Betti number of the manifold it is applied to.
	\end{theorem}
	
	While the star surgery operations allow for surgeries on much more general configurations of surfaces, it is not obvious whether such operations are equivalent to sequences of previously understood rational blow-downs (the configurations to be rationally blown-down may not be visible from the initial plumbing). In fact, for linear plumbings, whose boundaries are lens spaces, Bhupal and Ozbagci showed that all of the minimal convex fillings of these canonical contact lens spaces can be obtained from the plumbing via a sequence of the symplectic rational blow-downs of Fintushel-Stern and Park. Thus any symplectic surgery operation defined using linear plumbings is equivalent to a sequence of rational blow-downs. In light of these results, Ozbagci asked the author whether the same was true for star surgeries. While some of the star surgery operations can be decomposed into sequences of rational blow-downs (see \cite{KS} for examples), we show here that this is not the case in general.
	
	\begin{theorem}\label{thm:infinite}
	There are infinitely many star surgeries which are not equivalent to any sequences of generalized symplectic rational blow-downs.
	\end{theorem}
	
	This result implies that star surgery could potentially be used to construct manifolds inaccessible to rational blow-down techniques. However, it is unknown whether or not the closed manifolds obtained from star surgery in \cite{KS} are diffeomorphic to previous exotic examples constructed through rational blowdown because the Seiberg-Witten invariants agree in such examples. The relationship between exotic manifolds which have the same known invariants but are obtained from distinct constructions, is quite interesting but very mysterious. The star surgery adds to the set of examples on which this relationship can be studied.
	
	As with the rational blow-down and its generalizations, each star surgery operation can be interpreted as a \emph{monodromy substitution}: a pair of factorizations of a certain diffeomorphism of a planar surface into products of positive Dehn twists. Each such monodromy substitution is a relation in the planar mapping class monoid. While a complete presentation for the planar mapping class \emph{groups} is well understood, we are much further from understanding a complete set of relations for the planar mapping class \emph{monoids}. The monodromy substitutions corresponding to the generalized rational blow-down operations were found in \cite{EMV}.  The particular monodromy substitutions for the star surgery operations in theorem \ref{thm:infinite} are related to certain line arrangements as is seen in section \ref{s:notQbd}. These mapping class monoid relations agree with the corresponding relations from \cite{H} though the interpretation here is in terms of symplectic fillings while in \cite{H} the interpretation is in terms of braid monodromy. The two perspectives are closely related, but the advantage of the symplectic filling interpretation is that we are able to not only \emph{discover} such relations but also, to some extent, \emph{classify} these relations and their dependencies on each other. Theorem \ref{thm:infinite} can be interpreted as saying that this infinite family of relations in the planar mapping class monoid are not generated by the relations coming from the generalized rational blow-down monodromies.

\section{Background: classifying symplectic fillings}

	We briefly describe the strategy used in \cite{St} to classify symplectic fillings of Seifert fibered spaces with a canonical contact structure. We restrict to the Seifert fibered spaces which arise as the boundary of a star-shaped plumbing of spheres with $k$ arms, such that the Euler number on the central sphere satisfied the inequality $e_0\leq -k-1$ and all other Euler numbers are $\leq -2$. We call such a star-shaped plumbing \emph{dually positive}. These plumbings have convex symplectic structures by a result of Gay and Stipsicz \cite{GS2}.	Denote the induced contact boundary by $(Y,\xi)$. Then $(Y,\xi)$ has a concave cap given by another star-shaped plumbing of spheres, where here the central sphere has Euler number $+1$ as in figure \ref{fig:dualgraph}. This cap is the plumbing specified by the ``dual graph'' described by Stipsicz, Szab\'{o}, and Wahl in \cite{SSW} (the construction is described in the proof of theorem \ref{thm:maxE} and further details can also be found in section 2.1 of \cite{St}).
	
	We will refer to the star-shaped plumbing with convex boundary (whose central sphere has square $e_0\leq -k-1$) as the \emph{filling plumbing} or \emph{convex plumbing} and the star-shaped plumbing with concave boundary (whose central sphere has square $+1$) as the \emph{cap plumbing} or \emph{concave plumbing}. Note that the filling plumbing is the piece which is cut out during the star surgery. However, the cap plumbing is actually the main object which is studied in the classification proofs.
	
	We can glue any alternate convex filling to this concave cap plumbing to give a closed symplectic manifold, which by a classification theorem of McDuff \cite{Mc}, is necessarily a blow-up of $\CP2$ (because it contains a symplectic sphere of self-intersection $1$). Therefore all convex fillings of $(Y,\xi)$ embed into $\CP2\# N\barCP2$ (for some $N$) as the complement of the cap. The strategy is to classify all symplectic embeddings of the cap plumbing into $\CP2 \# N\barCP2$ up to isotopy.

	\begin{figure}
		\centering
		\includegraphics[scale=.6]{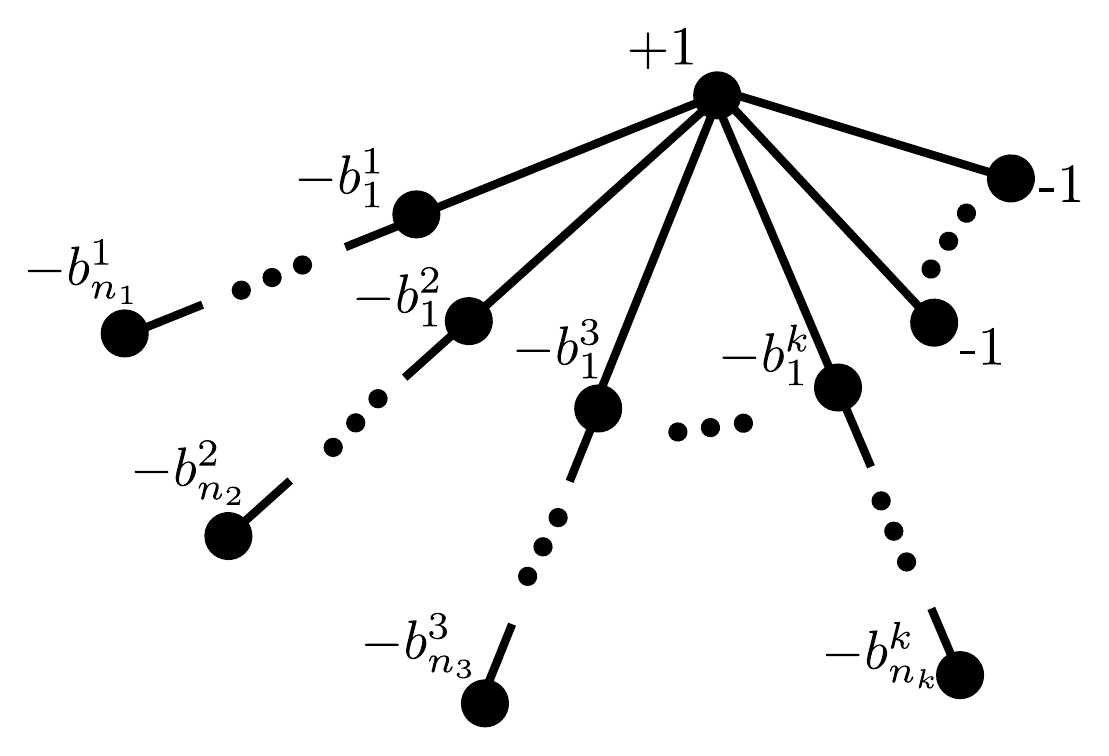}
		\caption{Dual graph plumbing corresponding to the concave cap. $b_i^j\geq 2$.}
		\label{fig:dualgraph}
	\end{figure}

	Classifying symplectic embeddings of the cap into $\CP2\# N\barCP2$ can be split into two steps. First classify \emph{homological embeddings}: the induced maps on second homology. Because the core spheres of this plumbing are symplectic, the adjuction formula and intersection relations can be used to significantly constrain the possible homological embeddings of a given plumbing. The second step is to determine the isotopy classes of embeddings inducing a given homological embedding. By carefully blowing down $J$-holomorphic exceptional spheres, this reduces to classifying \emph{symplectic line arrangements} where the combinatorics of the line arrangement is determined by the homological embedding. In many cases, one can prove there is a unique isotopy class for each homological embedding. After classifying the embeddings one can describe the complementary convex fillings as Lefschetz fibrations using Kirby calculus.

\section{Euler characteristic and star surgery}

	A star surgery operation replaces the symplectic filling given by a star-shaped plumbing of spheres, by an alternate symplectic filling of the same contact boundary. Therefore a star surgery will decrease the Euler characteristic of a manifold precisely when the Euler characteristic of the alternate symplectic filling is less than that of the plumbing.  Constructions of many symplectic 4-manifolds with small Euler characteristic have historically proven to be more difficult to find than those with large Euler characteristic. Therefore operations which reduce Euler characteristic and can be used to produce exotic 4-manifolds are of particular interest to 4-manifold topologists. The main goal of this section is to prove theorem \ref{thm:smaller} which is an immediate consequence of the following more precise results.
	
	\begin{theorem}
		Let $P$ be a dually-positive star-shaped convex plumbing of spheres inducing $(Y,\xi)$ on its boundary. Then the $P$ has larger Euler characteristic than any other minimal convex symplectic filling of $(Y,\xi)$. \label{thm:maxE} 
	\end{theorem}

	We know that all of the symplectic fillings we are considering will embed in a blow-up of $\CP2$ as the complement of the cap plumbing. Both $\CP2\#N\barCP2$ and the cap plumbing are simply connected. Moreover, the plumbings we consider here have Seifert-fibered boundaries which are rational homology spheres (the dually positive condition $e_0\leq -k-1$ ensures this). Because the cap plumbing $P'$ contains a sphere of positive square, $b_2^+(P')=b_2^+(\CP2\#N\barCP2)=1$. Therefore by the Mayer-Vietoris sequence we conclude the following results.
	
	\begin{prop}\label{prop:b1}
		Every convex symplectic filling $W$ of $(Y,\xi)$ as above has $b_1(W)=0$ and $b_2^+(W)=0$.
	\end{prop}
	
	In fact, the contact manifolds we consider have planar open books \cite{GM}  so every convex symplectic filling is deformation equivalent to a Stein domain \cite{W} and thus have $b_3=0$. It follows that the Euler characteristic of the filling determines the rank of its second homology. Examples of such fillings have varying fundamental groups. In examples, we have observed trivial and finite cyclic fundamental groups for fillings which are distinct from the convex plumbing.
	
	The Euler characteristic of any symplectic filling which arises as the complement of an embedding of the cap plumbing into $\CP2 \# N\barCP2$ is $2+N-|P'|$ where $|P'|$ is the number of spheres plumbed together in the cap. To find fillings without exceptional spheres, $N$ must be the minimal number of distinct exceptional classes appearing with non-zero coefficient in the homological embedding. This follows from Lemma 4.5 of \cite{L}, which proves that there is a $J$-holomorphic exceptional sphere disjoint from the cap plumbing spheres whenever there is an exceptional sphere with algebraic intersection $0$ with each sphere in the cap plumbing. Therefore to search for minimal symplectic fillings with small Euler characteristic, we want to look for homological embeddings of the cap which use a small number of exceptional classes relative to the size of the cap plumbing.
	
	There are significant restrictions on the homology classes represented by the images of the cap plumbing spheres under a symplectic embedding. McDuff's theorem identifies the $+1$ sphere in the cap with the complex projective line, so the image of its homology class is necessarily $h$. The adjunction formula and the intersection form on the cap plumbing imply the following lemma which is proven in \cite{St} section 2.4.

	Let $C_i^j$ be the image in $\CP2\#N\barCP2$ of the core plumbing sphere in the cap corresponding to a vertex in the plumbing graph which lies in the $j^{th}$ arm and is separated from the central $+1$ vertex by $i$ edges.
	
	\begin{lemma} \label{l:homform} The homology class of $C_1^j$ (intersecting the central sphere) has the form
		$$\left[C_1^j\right]=h-e_{m_1}^{1,j}-\cdots -e_{m_{n_{1,j}+1}}^{1,j}$$
	
	The homology class of $C_i^j$ for $i>1$ (not intersecting the central sphere) has the form
	$$\left[C_i^j \right]=e_{m_0}^{i,j}-e_{m_1}^{i,j}-\cdots - e_{m_{n_{i,j}-1}}^{i,j}$$
	\end{lemma}
	
	The following lemmas easily follow from the intersection relations and lemma \ref{l:homform}.
	
	\begin{lemma}\label{l:mixing}
		For each distinct pair $j,j'$, there is exactly one $e_x$ which appears with coefficient $-1$ in both $[C_1^j]$ and $[C_1^{j'}]$.
	\end{lemma}
	
	\begin{lemma}\label{l:first}
		The class of the exceptional sphere which appears with coefficient $+1$ in $[C_2^j]$ appears with coefficient $-1$ in sphere $[C_1^j]$ and no other exceptional class appears with non-zero coefficient in both.
	\end{lemma}
	
	\begin{lemma}\label{l:consecutive}
		For $i>1$, either the exceptional class with coefficient $+1$ in $[C_i^j]$ appears with coefficient $-1$ in $[C_{i+1}^j]$ or the exceptional class with coefficient $+1$ in $[C_{i+1}^j]$ appears with coefficient $-1$ in $[C_i^j]$, or both. In particular, they do not share the same exceptional class with $+1$ coefficient. Furthermore, the exceptional classes which appear with coefficients $-1$ in $[C_i^j]$ are disjoint from those which appear with coefficient $-1$ in $[C_{i+1}^j]$ unless both conditions in the first sentence are satisfied, in which case they share exactly one exceptional class with $-1$ coefficient in common.
	\end{lemma}
	
	\begin{lemma}\label{l:pos}
		If $e_x$ appears with coefficient $+1$ in $[C_i^j]$ then it does not appear with coefficient $+1$ in $[C_{i'}^{j'}]$ for any $(i',j')\neq (i,j)$.
	\end{lemma}
	
%
	%
	%
	%
	
	\begin{lemma}\label{l:share2}
		If $e_x$ appears with nonzero coefficient in distinct classes $[C_i^j]$ and $[C_{i'}^{j'}]$ and it is not the case that $(i,j)=(i'\pm 1,j')$ or that $i=i'=1$, then we have one or both of the following two possibilities:
		\begin{enumerate}
			\item the exceptional class with coefficient $+1$ in $[C_i^j]$ appears with coefficient $-1$ in $[C_{i'}^{j'}]$
			\item the exceptional class with coefficient $+1$ in $[C_{i'}^{j'}]$ appears with coefficient $-1$ in $[C_i^j]$
		\end{enumerate}
		If only one of these possibilities holds then there is exactly one exceptional class which appears with coefficient $-1$ in both $[C_i^j]$ and $[C_{i'}^{j'}]$. If both (1) and (2) hold, then there are exactly two exceptional classes which appear with coefficient $-1$ in both.
	\end{lemma}
	
	
	With all of these restrictions on the homology embeddings, the different possibilities for the Euler characteristics of the corresponding fillings are determined by the varying ways that exceptional classes appear with non-zero coefficients in the homology classes of distinct spheres.
	
	\begin{proof}[Proof of theorem \ref{thm:maxE}]
		
		We will first describe the homological embedding of the cap plumbing such that its complement is the convex plumbing. 
		
		The following process produces a decomposition of a blow-up of $\CP2$ into two star-shaped plumbings, one the convex filling and the other the concave cap. More details can be found in \cite{SSW} or \cite[section 2.1]{St}. We will call the corresponding embedding of the cap the \emph{canonical cap embedding} and the induced map on homology, the \emph{canonical homological embedding}. Start with $\CP2\#\barCP2$, viewed as a sphere bundle over a sphere. Keep track of two sections - the $0$-section which has square $+1$ and is a complex projective line representing the homology class $h$, and the $\infty$-section which has square $-1$ and is an exceptional sphere representing the homology class $e_1$. Also keep track of $-e_0-1$ fiber spheres (necessarily of square $0$) which represent $h-e_1$. Then blow-up at intersections between each of the tracked fibers with the $\infty$-section so that its proper transform has self-intersection $e_0$. Keep track of the exceptional spheres and all proper transforms. Continue blowing up at points which lie on the intersection of an exceptional $(-1)$-sphere with one of the other spheres being tracked (there are two points on each exceptional sphere to choose from and the varying choices yield varying dual pairs of star-shaped plumbings). After blowing up we add the new exceptional spheres into the set of spheres we track along with the proper transforms of the spheres we were tracking before. The blow-ups create singular fibers which can each be cut along the most recently introduced exceptional sphere so that on one side we have the arms of the convex plumbing emanating from the proper transform of the $\infty$-section, and on the other side we have the arms of the cap plumbing emanating from the $0$-section sphere of square $+1$.
		
		The homology classes of the cap spheres in this construction can be easily computed. The spheres adjacent to the central $+1$ sphere represent the class $h-e_1-e_{i_1}-\dots - e_{i_n}$. They all share $e_1$ with coefficient $-1$, but the other $e_x$'s that appear with nonzero coefficient are all distinct. Subsequent spheres in the arms are the proper transforms of the exceptional sphere which intersects the previous sphere in that arm. Therefore, the spheres in the $j^{th}$ arm represent homology classes as follows.
		$$h-e_1-e_1^{1,j}-\cdots - e_{n_{1,j}}^{1,j}$$
		$$e_{n_{1,j}}^{1,j}-e_1^{2,j}-\cdots - e_{n_{2,j}}^{2,j}$$
		$$e_{n_{2,j}}^{2,j}-e_1^{3,j}-\cdots - e_{n_{3,j}}^{3,j}$$
		$$\vdots$$
		$$e_{n_{m-1,j}}^{m-1,j}-e_1^{m,j}-\cdots - e_{n_{m,j}}^{m,j}$$
		Here all $e_x^{i,j}$ are exceptional classes distinct from each other and from $e_1$. There are no exceptional sphere classes which appear with nonzero coefficient in more than one arm because the blow-ups are all done in distinct singular fibers which each correspond to distinct arms. The only exceptional classes besides $e_1$ that appear with nonzero coefficient in two different spheres are in adjacent spheres, and appear with coefficient $-1$ in the inner-more sphere and with coefficient $+1$ in the outer-more sphere.
		
		Now we will show that any other homological embedding of the cap plumbing uses strictly fewer distinct exceptional classes than the canonical homological embedding.
		
		First we focus on the homology classes of the spheres adjacent to the center $C_1^j$ (we will drop the brackets indicating homology class for readability). By Lemma \ref{l:mixing}, each pair of spheres intersecting the central $+1$ sphere in the cap, must have exactly one shared $e_i$ appearing with $-1$ coefficient in both. In the canonical embedding, they all share the same class, $e_1$. In an embedding where the $C_1^j$ did not all share the same class, there would necessarily be at least one sphere $C_1^{j_0}$ in which two exceptional classes $e_x$ and $e_y$ appear with $-1$ coefficient, where $e_x$ appears with $-1$ coefficient in $C_1^{j_1},\cdots , C_1^{j_n}$ and $e_y$ appears with $-1$ coefficients in a disjoint set of spheres $C_1^{j_{n+1}}, \cdots , C_1^{j_m}$. There is necessarily a third exceptional class $e_z$ which appears with $-1$ coefficient in $C_1^{j_1}$ and $C_1^{j_{n+1}}$. Now compare this to the homology embedding where $C_1^{j_0},C_1^{j_1},\cdots , C_1^{j_m}$ all share the same $e_z$ with coefficient $-1$ so $e_x$ and $e_y$ are eliminated. Then because the squares of the homology classes of $C_1^{j_0},C_1^{j_1},\cdots, C_1^{j_n},C_1^{j_{n+1}},\cdots, C_1^{j_m}$ are fixed, there must be new distinct exceptional classes $e_a,e_{b_1},\cdots, e_{b_m}$, with one appearing with coefficient $-1$ in each of these classes. Therefore decreasing the number of distinct $e_i$'s which are shared between the spheres adjacent to the central vertex, increases the total number of distinct exceptional classes appearing with nonzero coefficients in the embedding (we eliminated two exceptional class at the cost of adding $1+m\geq 3$ new ones).
		
		Next we consider classes of the spheres $C_i^j$ for $i>1$. In the canonical cap embedding, the only exceptional classes which appear with non-zero coefficient in more than one class other than the $C_1^j$ occurs between consecutive spheres within an arm where the class has coefficient $-1$ in $C_i^j$ and $+1$ in $C_{i+1}^j$. By Lemma \ref{l:consecutive}, there is always an exceptional class appearing with non-zero coefficient in both consecutive spheres in a given arm, but it is possible that it appears with coefficient $+1$ in $C_i^j$ and $-1$ in $C_{i+1}^j$. We claim that in this scenario there is necessarily some exceptional class which appears with non-zero coefficient in two classes of spheres which are not consecutive. Let $i_0$ be the largest value less than $i$ for which the exceptional class with coefficient $1$ in $C_{i_0+1}^j$ appears with coefficient $-1$ in $C_{i_0}^j$. By Lemma \ref{l:first}, $i_0$ is well-defined. Then for all $i_0<p\leq i$ there is an exceptional class $e_{k_p}$ with coefficient $1$ in $C_p^j$ and coefficient $-1$ in $C_{p+1}^j$. Therefore $e_{k_{i_0+1}}$ has coefficient $-1$ in $C_{i_0}$ and $C_{i_0+2}^j$, proving the claim.
		
		We conclude using Lemma \ref{l:share2} and the previous paragraph that any homological embedding that differs from the canonical embedding has a larger number of exceptional classes which appear more than once with non-zero coefficient. Because the classes of the cap spheres have the restricted form of Lemma \ref{l:homform} and the squares of those homology classes are fixed, the total number of exceptional classes which appear with non-zero coefficient (counted with multiplicity) is fixed. Thus the total number of \emph{distinct} exceptional classes is maximized when the number of exceptional classes appearing with non-zero coefficient more than once is minimized. This implies that the canonical homological embedding has the maximal number of distinct exceptional classes appearing with non-zero coefficient. Therefore the convex plumbing has the maximal Euler characteristic of any minimal symplectic filling of its contact boundary. To show that it is the unique filling realizing this maximum, we must verify that there is a unique isotopy class of embeddings of the cap plumbing inducing the canonical homological embedding. This follows from the arguments of \cite{St} Lemma 2.7 and 2.8 which we discuss in a more general form in the following section.
	\end{proof}
	
\section{Unique isotopy classes}
	
	 We discuss the isotopy classes of embeddings of the cap plumbing within a fixed homological embedding. While similar statements can be found in \cite{St} section 2.5, here we prove that there is a unique non-empty isotopy class for a larger class of homological embeddings (which is needed for the infinite family in the next section), and we refine the argument so that the isotopy equivalence is through symplectic configurations instead of only smooth ones. This allows us to conclude classification results of minimal fillings up to symplectic deformation equivalence and symplectomorphism instead of only up to diffeomorphism.
	 
	 Suppose we have an embedding of the cap plumbing $P'$ into $\CP2\#N\barCP2$. Choose an almost complex structure $J$ where the core spheres of the cap plumbing are $J$-holomorphic, and blow down $J$-holomorphic exceptional spheres until reaching $\CP2$. Then the homology classes of the spheres determine their images in $\CP2$ after blowing down. In particular, all of the spheres disappear except for those whose homology classes have non-zero coefficient for $h\in H_2(\CP2)$. Because of lemma \ref{l:homform}, this means we are left with the central $+1$ sphere, together with the images of $\{C_1^j\}_{j=1}^d$ under blowing-down, each of which becomes a symplectic sphere $\overline{C_1^j}$ in the class of the line $h$. The $J$-holomorphic exceptional spheres are disjoint from the central $+1$ sphere so its double-point intersections with $\overline{C_1^1},\cdots, \overline{C_1^d}$ are unaffected by blowing down, but the images any pair of the $\overline{C_1^j}$ in $\CP2$ intersect transversally in a single point, and collections $\overline{C_1^{j_1}},\cdots, \overline{C_1^{j_p}}$ intersect at a common intersection point precisely when the homology classes $[C_1^{j_1}],\cdots, [C_1^{j_p}]$ all shared the same exceptional class with coefficient $-1$ in the embedding into $\CP2\#N\barCP2$. Thus after blowing-down, we are left with a \emph{symplectic line arrangement} (a collection of symplectic spheres in $\CP2$ each in the homology class of the line which intersect pairwise once, but not necessarily generically). The combinatorics describing the point-line incidences of the symplectic line arrangement is determined by the homological embedding of the $C_1^j$ (specifically which classes had shared $e_i$). Note that in the plumbing embedding, because all the $[C_1^j]$ share the same exceptional class, after blowing down they will represent symplectic lines which all intersect at a single point (the central sphere will represent a symplectic line which intersects the others generically and all other spheres disappear under blowing down). This is the simplest case which is covered by the isotopy classifications in \cite{St} and the refinements below. 
	 
	 We say that an intersection point in a symplectic line arrangement is a \emph{multi-intersection} if there are $\geq 3$ symplectic lines passing through it. Note that for any symplectic line arrangement there is an almost complex structure tamed by the symplectic form making all the lines $J$-holomorphic because all the symplectic spheres intersect each other positively. (Standard arguments hold for configurations of symplectic surfaces which intersect positively and generically, and this can be extended to the present case by blowing up at the multi-intersections and then choosing $J$ and blowing back down.) We show that if there are no more than two multi-intersections on each line, there is a unique non-empty isotopy class of symplectic line arrangements.
	 
	 			\begin{prop}\label{l:JtoCx}
	 				Suppose $C^0,C^1,\cdots, C^d$ form a symplectic line arrangement which is $J_0$-holomorphic for some $J_0$ tamed by $\omega_{std}$ such that no $C^j$ contains more than two multi-intersection points. Then the spheres $C^0,C^1,\cdots, C^d$ can be isotoped to $J_1$-holomorphic lines for any $J_1$ tamed by $\omega_{std}$ through symplectic spheres such that the combinatorial intersection data of the arrangement remains unchanged throughout the isotopy.
	 			\end{prop}
	 			
	 			 In particular letting $J_1$ be the standard complex structure, any symplectic line arrangement where each line contains at most two multi-intersections is isotopic to a complex line arrangement.
	 
	 			\begin{proof}
	 				Because $J_0$ and $J_1$ are both tamed by $\omega_{std}$, and the space of such $J$ is contractible and thus connected, there exists an interpolating family $\{J_t\}$. A theorem of Gromov \cite{Gromov} states that for any $J$ tamed by the standard symplectic structure on $\CP2$, any two points $v_1\neq v_2\in \CP2$ lie on a unique non-singular rational (i.e. diffeomorphic to $S^2$) $J$-holomorphic curve homologous to $\CPo\subset \CP2$. By fixing two points on each sphere and considering the family of $J_t$-holomorphic spheres through those two points, we get an isotopy $C^0_t,C^1_t,\cdots , C^d_t$ from the original embedded $J$-holomorphic spheres to complex projective lines. During this isotopy we fix exactly two points on each sphere. If at a multi-intersection point, we choose to fix that point on every sphere passing through it, then that intersection is preserved (though potentially other spheres may pass through that point during the isotopy). Using the hypothesis, we will fix the multi-intersection points in this way. Thus we can choose the isotopy so that at worst the intersection configuration becomes more degenerate. We will now discuss how to modify this isotopy to one which preserves the intersection configuration throughout.
	 				
	 				The spheres which have two fixed multi-intersection points are rigidly determined by $J_t$ and the placement of those multi-intersection points. We begin the isotopy by making small perturbations of the multi-intersection points to avoid degenerations of the intersections between these rigid spheres. The first phenomenon which we need to avoid, is that three of the multi-intersection points could become $J_t$-collinear for some $t>0$, meaning that they all lie on the same $J_t$-holomorphic line. The second phenomenon we need to avoid is a generic intersection between two of the rigid spheres could pass onto a third rigid sphere. Choose an ordering of the multi-intersection points, $p_1,p_2,p_3,\cdots, p_z$. Leaving $p_1$ and $p_2$ fixed, consider the union of the $J_t$-holomorphic lines through $p_1$ and $p_2$ over all $0\leq t\leq 1$. This subset of $\CP2$ will have at least codimension one. Therefore, if $p_3$ lies in this subset, it can be perturbed slightly to $p_3'$ so that it no longer is $J_t$-collinear with $p_1$ and $p_2$ for any $t$. Define an isotopy of the curves $C^0,C^1,\cdots, C^d$ by choosing a path $\gamma(s)$ from $p_3$ to $p_3'$. For each sphere $C^j$ passing through $p_3$, isotope $C^j$ by defining $C^j_s$ as the unique $J_0$-holomorphic sphere through the other chosen fixed point which is not $p_3$ and $\gamma(s)$. For each sphere that does not pass through $p_3$, fix it throughout this isotopy. By choosing the perturbation $p_3'$ sufficiently close to $p_3$, we can ensure that the incidences of $C^0,C^1,\cdots, C^d$ are unchanged during this small isotopy. Next consider the space of $J_t$ lines through $p_1$ and $p_2$, $p_2$ and $p_3$, and $p_1$ and $p_3$ for all $t$. Again this subspace has at least codimension one, so $p_4$ can be perturbed to avoid this subspace and a corresponding isotopy of the $C^j$ can be defined to preserve the incidences. Now some additional points of interest appear. Let $q^{i,j,k,l}_t$ denote the intersection between the $J_t$ line through $p_i$ and $p_j$ and the $J_t$ line through $p_k$ and $p_l$. We proceed to isotope $p_N$ for $N\geq 5$ as before, except we want to also avoid an additional codimension one subspace to ensure that lines through the various $p_a$ avoid the $q^{i,j,k,l}_t$. The space of points which lie on a $J_t$ line through one of the $q^{i,j,k,l}_t$ and one $p_a$ for $a<N$ and some $t\in [0,1]$ has at least codimension one so $p_N$ can be perturbed an arbitrarily small amount to avoid it.
	 				
	 				Now, after any needed isotopies in the previous step, append the isotopy $C^0_t$, $C^1_t, \cdots, C^d_t$ defined by fixing two points on each of (the newly isotoped) $C^0,C^1,\cdots, C^d$, including all of the multi-intersection points and considering the unique $J_t$-holomorphic line through those two points. Suppose $C^{j_0}_t$ passes through the intersection of spheres $C^{j_1}_t,\cdots, C^{j_p}_t$ for $t\in \mathbf{T}\subset I$ ($p\geq2$) when it shouldn't according to the original intersection configuration. Parameterize a neighborhood of this intersection point on $C^{j_1}_t$ by a small disk centered on the intersection point by $\eta(r,\theta)$.  By the previous paragraph $C^{j_0}_0$ passes through at most one multi-intersection. Let $p$ denote a point fixed on $C^{j_0}$ away from $C^{j_1}_t$. Define $C^{j_0}_{t,r,\theta}$ to be the unique $J_t$-holomorphic sphere through the point $\eta(r,\theta)$ and $p$. By choosing $\varepsilon$ sufficiently small, we can ensure that no new degeneracies of the configuration are introduced away from the degeneracy we are focusing on. For each fixed $t$ and corresponding almost complex structure $J_t$ there is a unique point on $C^{j_1}_t$ which we are trying to make $C^{j_0}_t$ avoid: its intersection with $C^{j_2}_t,\cdots, C^{j_n}_t$. Thus the degenerate configurations occur in a codimension two subset of the cylinder $(t,r,\theta)$, so we can choose a path which avoids this.
	 				
	 				Morally, we can perturb the line in a real $2$-dimensional space (because we are only trying to fix at most one multi-intersection point on the complex line) and the degeneracies we want to avoid are $0$-dimensional, so we can find a 1-parameter family of configurations which avoid these degeneracies. We repeat this for other spheres contributing to degeneracies and eventually find an isotopy from the $J_0$-holomorphic configuration we started with (at $r=t=0$) to a $J_{std}$-holomorphic (complex) configuration (at $t=1$).
	 				
	 			\end{proof}
	 			
	 			Finally, we show there is a unique isotopy class of $J$-holomorphic configurations in given combinatorial arrangements of the types covered by Proposition \ref{l:JtoCx}.
	 			
	 			\begin{prop} Fix an almost complex structure $J$ on $\CP2$ tamed by the standard symplectic structure. Then the space of $J$-holomorphic lines with a fixed incidence configuration with the property that no line contains more than two multi-intersection points is path-connected and non-empty. \label{lemma:configconnected} \end{prop}
	 			
	 			\begin{proof}
	 				We use the fact that there is a unique $J$-holomorphic line through any two points \cite{Gromov}. The configuration space of possibilities for the multi-intersection points $p_1,\cdots, p_z$ is parameterized by the complement of a co-dimension $\geq 2$ subset in $(\CP2)^{\times z}$ (as in the last part of the proof of the previous lemma but without the additional $t$ parameter). For fixed $p_1,\cdots , p_z$, the lines through two of the $p_i$ are determined so the space of possibilities for these lines is a single point. The space of possibilities for $J$ lines passing through one (respectively none) of the $p_i$ is non-empty, path-connected and real $2$-dimensional (respectively real $4$-dimensional), and the more degenerate configurations have co-dimension at least $2$.
	 			\end{proof}

\section{An infinite family of star surgeries inequivalent to any sequence of rational blowdowns}
\label{s:notQbd}
	\begin{figure}
		\centering
		\captionsetup[subfigure]{width=1.5in}
		\subfloat[Filling graph]{\label{fig:gex1}\includegraphics[scale=.6]{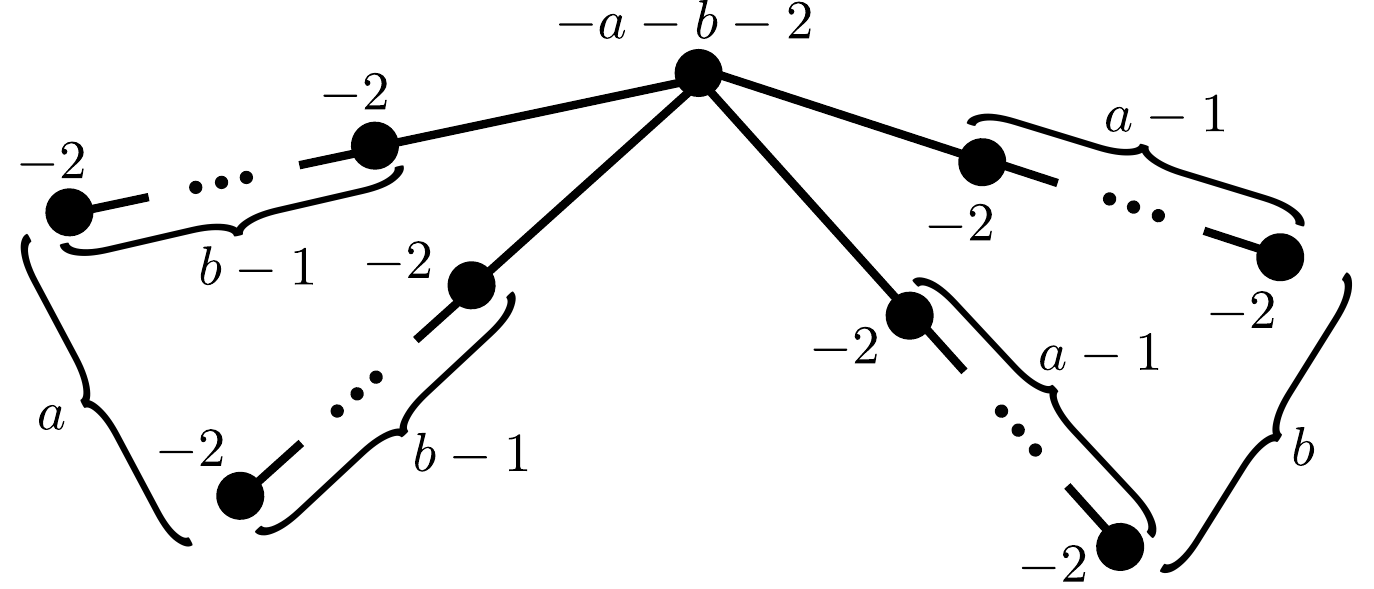}}\hspace{.5in}
		\subfloat[Cap graph]{\label{fig:dgex1}\includegraphics[scale=.6]{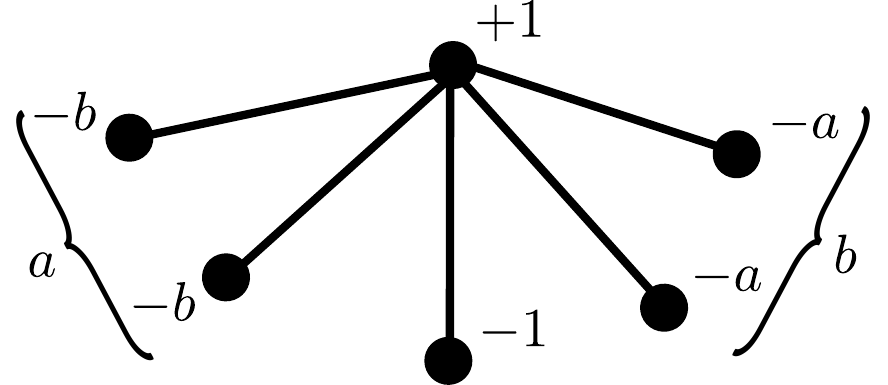}}
		\caption{Graphs $\Gamma_{a,b}$ for the convex plumbing and $D\Gamma_{a,b}$ for the cap plumbing.}
	\end{figure}
	
	In this section, we will show that the plumbings according to the graphs in figure \ref{fig:gex1} each have a star-surgery operation which for infinitely many values of $a,b$ we can show is not equivalent to any sequence of symplectic rational blow-downs (where here symplectic rational blow-down includes Fintushel and Stern's original family, Park's generalization, as well as the further negative definite examples classified in \cite{SSW} and \cite{BS}). This is in contrast to the results of \cite{BO} which show that all fillings of lens spaces are obtained from a linear plumbing by a sequence of rational blow-downs. The following is a more precise version of theorem \ref{thm:infinite}.

	\begin{theorem}
		The convex boundary $(Y,\xi_{can})$ of the plumbing of spheres $P_{a,b}$, plumbed according to the graph $\Gamma_{a,b}$ in figure \ref{fig:gex1}, has exactly two minimal strong symplectic fillings. One is the plumbing itself and the other has Euler characteristic $ab-a-b+2$. If $ab+1$ is not divisible by $a+b$, then the filling of Euler characteristic $ab-a-b+1$ cannot be obtained from the plumbing by any sequence of symplectic rational blow-downs.
	\end{theorem}
	
		Note that $ab+1$ is divisible by $a+b$ if and only if $b^2-1$ is divisible by $a+b$ since $ab+1=(a+b)b-(b^2-1)$. In particular, if $a>b^2-b-1$ the condition holds, so there are infinitely many pairs $(a,b)$ which satisfy this condition.

	\begin{proof}
	
		The graph for the cap plumbing is given in figure \ref{fig:dgex1}. Let $C_0$ denote the corresponding central $+1$ sphere, $C_1,\cdots, C_a$ denote the $a$ $(-b)$-spheres, $C_{a+1}$ the unique $(-1)$-sphere and $C_{a+2},\cdots, C_{a+b+1}$ the $b$ $(-a)$-spheres. The existence of the $(-1)$ sphere in the cap strongly restricts the homological embedding possibilities. There are only two possible homological embeddings satisfying Lemmas \ref{l:homform} and \ref{l:mixing}. One is the canonical homological embedding for the complement of the convex plumbing. The other is induced by the proper transform of the blow-up of the line arrangement in figure \ref{fig:lines} given as follows:
		
		\begin{eqnarray*}
			[C_0] &=& h\\
			\left[C_1\right] &=& h-e_1-e_{1,1}-\cdots - e_{1,b}\\
			 &\vdots&\\
			\left[C_a\right] &=& h-e_1-e_{a,1}-\cdots - e_{a,b}\\
			\left[C_{a+1}\right]&=& h-e_1-e_2\\
			\left[C_{a+2}\right]&=& h-e_2-e_{1,1}-\cdots - e_{a,1}\\
			&\vdots&\\
			\left[C_{a+b+1}\right]&=& h-e_2-e_{1,b}-\cdots - e_{a,b}\\
		\end{eqnarray*}
		
		\begin{figure}
			\centering
			\includegraphics[scale=.4]{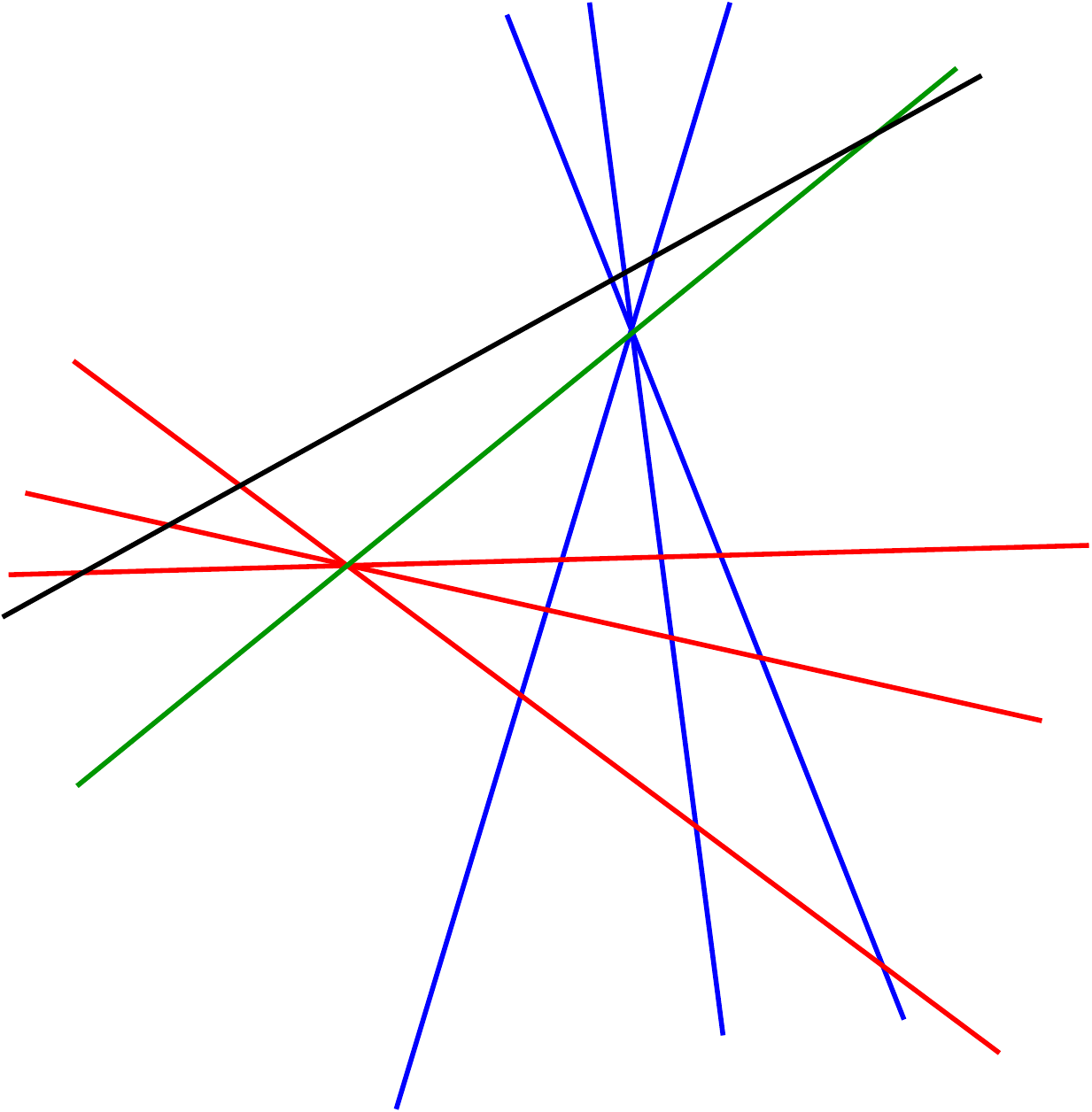}
			\caption{A line configuration with $a$ red lines intersecting at a point $p$, $b$ blue lines intersecting at a point $q$, one green line passing through both $p$ and $q$ and one generic black line. After blowing up at the the intersections between the blue, red, and green lines we get a homological embedding of the $D\Gamma_{a,b}$ plumbing where the proper transforms of the red lines become $C_1,\cdots, C_a$, the green becomes $C_{a+1}$ and the blue become $C_{a+2},\cdots, C_{a+b+1}$.}
			\label{fig:lines}
		\end{figure}
		
		Because there are only two exceptional classes which appear with non-zero coefficient in more than two sphere classes, Propositions \ref{l:JtoCx} and \ref{lemma:configconnected} imply there is at most one isotopy class of symplectic embeddings for each of these two homological embeddings. Therefore there are at most two symplectic deformation classes of convex symplectic fillings of the canonical contact boundary of this plumbing. The above homological embedding involves $ab+2$ exceptional classes, so the Euler characteristic of the complement to such an embedding is $ab-a-b+2$
		
		A smooth embedding of the cap plumbing realizing this homological embedding into $\CP2 \# (ab+2)\barCP2$ is given in figure \ref{fig:Emb}, along with a diagram for the complement of this embedding. This handlebody diagram naturally supports a Lefschetz fibration whose fibers are $a+b+1$ holed disks and whose vanishing cycles are given by the attaching curves for the $-1$ framed $2$-handles (see conventions below). The boundary open book decomposition can be shown to agree with the boundary open book decomposition of the canonical Lefschetz fibration for the plumbing graph of figure \ref{fig:gex1} by Lemma \ref{lem:MCGe0neg}.
		
		\subsection{Conventions on monodromy factorizations and Lefschetz fibrations}
		{	
		Let $\phi_{c}$ denote a positive Dehn twist around $c$. The product $\phi_{c_1}\phi_{c_2}\cdots \phi_{c_n}$ means first Dehn twist along $c_1$, then $c_2$, and so on until finally along $c_n$. When the fiber is a disk with holes, we can place the holes along a circle concentric with the bounday of the disk. Labeling the holes $\{1,\cdots, m\}$ counterclockwise, we use the notation $\phi_{i_1,\cdots, i_k}$ for $i_1,\cdots, i_k\in \{1,\cdots , m\}$ to indicate a positive Dehn twist about a curve which convexly contains the holes $i_1,\cdots, i_k$.
		
		Planar Lefschetz fibrations have a natural handle decomposition where the holes are represented by dotted circles forming a trivial braid corresponding to 1-handles and the vanishing cycles correspond to 2-handles. We view the holed-disk fibers as orthogonal to the dotted circles, oriented so that the outward normal defining orientation points downward. Then the monodromy factorization $\phi_{c_1}\cdots \phi_{c_n}$ corresponds to the Lefschetz fibration where the vanishing cycle $c_1$ appears at the top and $c_n$ at the bottom (though these vanishing cycles lie on the upside-down disk). To draw the handlebody, we will isotope the holes on the disk so that they all lie on the bottom half of the disk along a circle concentric to the boundary. An example, using the top to bottom convention where the outward normal to the disk points downward, is in figure \ref{fig:LFconv}.
		
		\begin{figure}
			\centering
			\includegraphics[scale=.7]{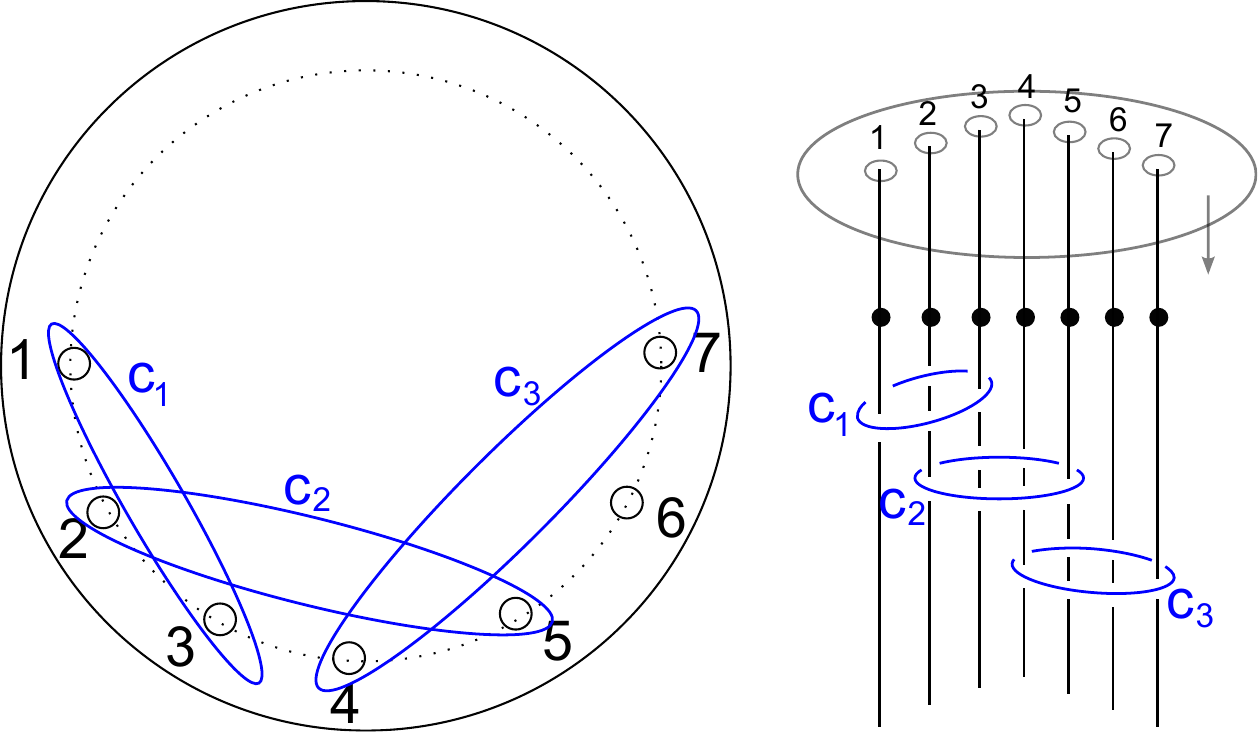}
			\caption{The Lefschetz fibration corresponding to the monodromy factorization $\phi_{1,3}\phi_{2,5}\phi_{4,7}=\phi_{c_1}\phi_{c_2}\phi_{c_3}$.}
			\label{fig:LFconv}
		\end{figure}
		
		If $A$, $B$, and $C$ are collections of holes ordered counter-clockwise on the disk, the \emph{lantern relation} is 
		$$\phi_{A\cup B\cup C}\phi_A\phi_B\phi_C=\phi_{A\cup B}\phi_{A\cup C}\phi_{B\cup C}.$$ 
		Similarly for $B_0,B_1,\cdots, B_p$ ordered counterclockwise, the \emph{daisy relation} (defined in \cite{EMV}), is
		$$\phi_{B_0\cup B_1\cup \cdots \cup B_p}\phi_{B_0}^{p-1}\phi_{B_1}\cdots \phi_{B_p} = \phi_{B_0\cup B_1}\phi_{B_0\cup B_2}\cdots \phi_{B_0\cup B_p}\phi_{B_1\cup \cdots \cup B_p}.$$
		
		Finally, sequences of daisy relations yield the \emph{generalized lantern relation}:
		$$\phi_{1,\cdots, k}\phi_{1}^{k-2}\cdots \phi_{k}^{k-2} = \phi_{1,2}\phi_{1,3}\cdots \phi_{1,k}\phi_{2,3}\cdots \phi_{2,k}\cdots \phi_{k-2,k-1}\phi_{k-2,k}\phi_{k-1,k}$$
		}

		\begin{figure}
			\centering
			\captionsetup[subfigure]{width=1.3in}
			\subfloat[Embedding of cap plumbing.]{\label{fig:Emb}\includegraphics[scale=.45]{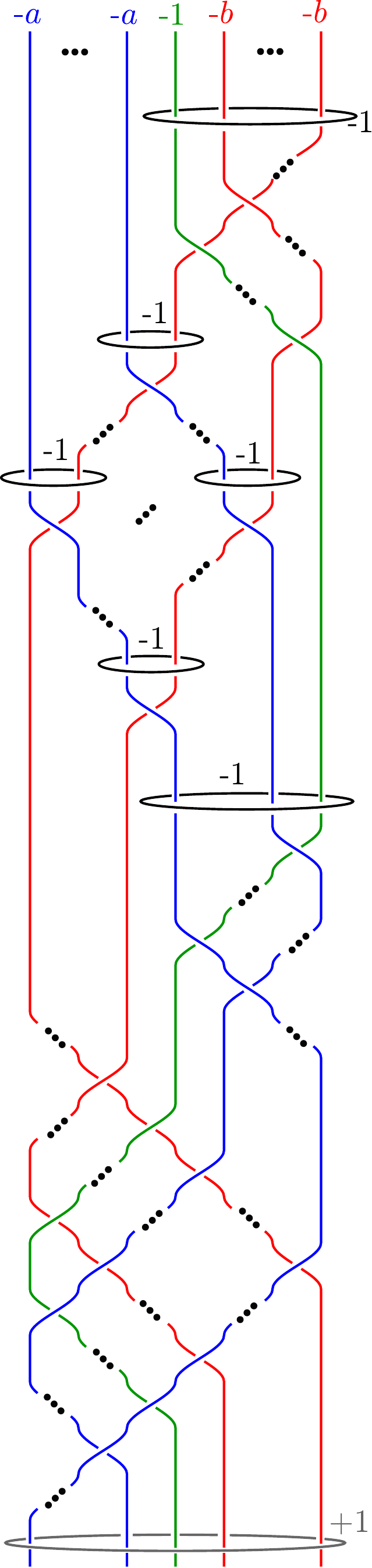}}\hspace{1cm}
			\subfloat[Complement of cap plumbing.]{\label{fig:Emb2}\includegraphics[scale=.45]{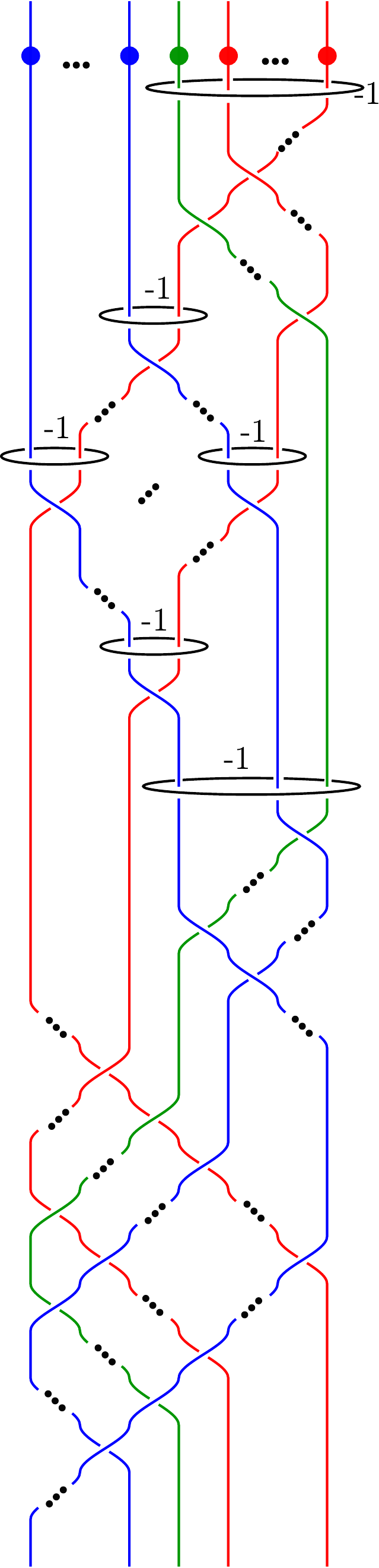}}\hspace{1cm}
			\subfloat[Lefschetz fibration for complement.]{\label{fig:LF}\includegraphics[scale=.7]{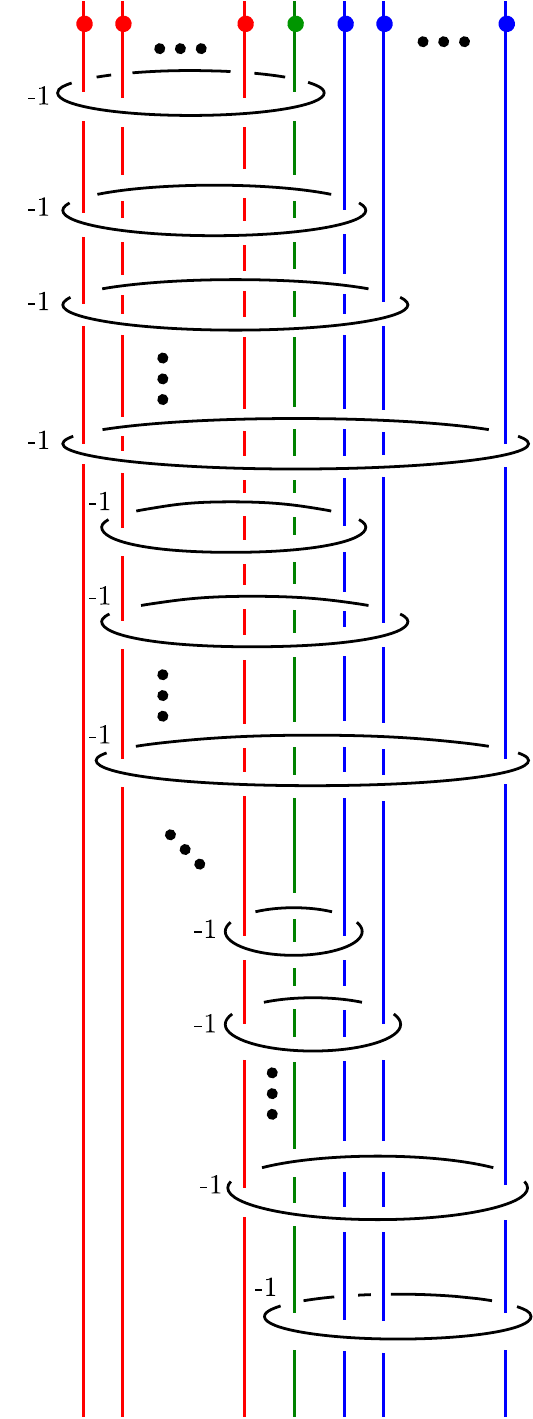}}
			\caption{The leftmost diagram represents a handlebody decomposition for $\CP2\#(ab+2)\barCP2$ with $(a+b+1)$ $3$-handles cancelling the $(a+b+1)$ $2$-handles whose attaching circles are the strands of the braid. The cores of the 2-handles attached along the  braid strands and the $+1$ framed circle, together with their disk Seifert surfaces form spheres in the cap plumbing. The center diagram is a handlebody for the complement of these embedded spheres (with no $3$- or $4$-handles) obtained by removing the braid $2$-handles and the $0$-handle, turning the resulting relative handlebody upsidedown, and simplifying the diagram by blowing down surgery curves on the lower boundary. Note that in this process the $(a+b+1)$ $3$-handles become the $(a+b+1)$ $1$-handles represented by dotted circles. The final diagram is an isotoped version of the center one which makes the Lefschetz fibration structure apparent. It is obtained from the previous diagram by rotating the plane of projection.}
			\label{fig:Embedding}
		\end{figure}

					\begin{lemma}\label{lem:MCGe0neg}
						Label the holes of the $m+1+n$-holed disk $A_1,\cdots, A_m$, $B$, and $C_1,\cdots, C_n$ counterclockwise. The following two products of positive Dehn twists are equal in the mapping class group of $\Sigma_{0,a+1+b}$
						$$F_{m,n}=\phi_{A_1,\cdots, A_m,B,C_1,\cdots, C_n}\phi_{A_1}^n\cdots \phi_{A_m}^n\phi_B\phi_{C_1}^m\cdots \phi_{C_n}^m$$
						
						$$G_{m,n}=\phi_{A_1,\cdots, A_m,B}(\phi_{A_1C_1}\cdots \phi_{A_1C_n})\cdots (\phi_{A_mC_1}\cdots \phi_{A_mC_n})\phi_{B,C_1,\cdots, C_n}$$
					\end{lemma}
					
					\begin{proof}
						
						Applying the generalized lantern relation to the factorization for $F_{m,n}$, we get a new factorization (not positive) given by
						$$\phi_{A_1}^{-m+1}\cdots \phi_{A_m}^{-m+1}\phi_B^{-m-n+2}\phi_{C_1}^{-n+1}\cdots \phi_{C_n}^{-n+1}(\phi_{A_1,A_2}\cdots \phi_{A_1,A_m})\phi_{A_1,B}(\phi_{A_1,C_1},\cdots \phi_{A_1,C_n})$$
						$$(\phi_{A_2,A_3}\cdots \phi_{A_2,A_m})\phi_{A_2,B}(\phi_{A_2,C_1},\cdots \phi_{A_2,C_n})\cdots (\phi_{A_{m-1},A_m})\phi_{A_{m-1},B}(\phi_{A_{m-1},C_1},\cdots \phi_{A_{m-1},C_n})$$
						$$\phi_{A_m,B}(\phi_{A_m,C_1}\cdots \phi_{A_m,C_n})(\phi_{B,C_1}\cdots \phi_{B,C_n})(\phi_{C_1,C_2}\cdots \phi_{C_1,C_n})\cdots (\phi_{C_{n-2}C_{n-1}}\phi_{C_{n-2},C_n})\phi_{C_{n-1},C_n}$$
						
						On the other hand, applying the generalized lantern relation twice on the factorization for $G_{m,n}$ to split the twists $\phi_{A_1,\cdots, A_m,B}$ and $\phi_{B,C_1,\cdots, C_n}$ shows that $\psi_{m,n}$ is equal to the product:
						$$\phi_{A_1}^{-m+1}\cdots \phi_{A_m}^{-m+1}\phi_B^{-m+1}(\phi_{A_1,A_2}\cdots \phi_{A_1,A_m})\phi_{A_1,B}(\phi_{A_2,A_3}\cdots \phi_{A_2,A_m})\phi_{A_2,B}\cdots (\phi_{A_{m-1},A_m})\phi_{A_{m-1},B}$$
						$$\phi_{A_m,B}(\phi_{A_1,C_1}\cdots \phi_{A_1,C_n})(\phi_{A_2,C_1}\cdots \phi_{A_2,C_n})\cdots(\phi_{A_{m-1},C_1}\cdots \phi_{A_{m-1},C_n})(\phi_{A_m,C_1}\cdots \phi_{A_m,C_n})$$
						$$\phi_{B}^{-n+1}\phi_{C_1}^{-n+1}\cdots \phi_{C_n}^{-n+1}(\phi_{B,C_1}\cdots \phi_{B,C_n})(\phi_{C_1,C_2}\cdots \phi_{C_1,C_n})\cdots (\phi_{C_{n-2},C_{n-1}}\phi_{C_{n-2},C_{n}})\phi_{C_{n-1},C_n}$$
						
						Comparing these two factorizations, we see that they differ only by commuting Dehn twists about curves which are disjoint. Specifically, all the twists about boundary parallel curves commute with anything, and the products of the form $(\phi_{A_i,C_1}\cdots \phi_{A_i,C_n})$ commute with Dehn twists about curves convexly enclosing any collection of the holes $\{A_{i+1}, \cdots, A_{m},B\}$.
					\end{proof}
					
		\begin{remark}
			Notice that the intermediate factorizations between $F_{m,n}$ and $G_{m,n}$ involve negative Dehn twists. Therefore these intermediate factorizations do not have geometric interpretations as symplectic fillings. The classification indicates that in fact there are no intermediate positive factorizations with such geometric interpretations.
		\end{remark}		
	
		We conclude that the plumbing $P_{a,b}$ and the filling $L_{a,b}$ described by the Lefschetz fibration of figure \ref{fig:LF} provide a complete list of symplectic fillings of their common contact boundary. Next we show that it is not possible to obtain $L_{a,b}$ from $P_{a,b}$ from any sequence of sympletic rational blow-downs when $a+b$ does not divide $ab+1$.
		
		Any sequence of symplectic rational blow-downs on the plumbing $P_{a,b}$ would produce a sequence of minimal symplectic fillings of the same contact boundary (since the rational blow-downs are symplectic operations that are performed on the interior). Since there are only two minimal symplectic fillings, the only way to obtain $L_{a,b}$ from $P_{a,b}$ by a sequence of rational blow-downs is by a single rational blow-down. The change in the Euler characteristic of a manifold before and after a rational blow-down is precisely the number of spheres in the rational blow-down graph. The Euler characteristic of $P_{a,b}$ is $a(b-1)+b(a-1)+1+1=2ab-a-b+2$, and the Euler characteristic $L_{a,b}$ is $ab-a-b+2$. Therefore it suffices to show that there is no plumbing with $ab$ vertices which can be rationally blown-down and which embeds into $P_{a,b}$.
		
		The classification of star-shaped plumbings which can be rationally blown-down was completed by Bhupal and Stipsicz \cite{BS}, and it was shown in \cite{PShS} that no other plumbings admit rational disk smoothings. The graphs which can be rationally blown-down are either linear or star-shaped with three or four arms. Using the list of Bhupal and Stipsicz, we observe that in each three or four armed graph which can be rationally blown-down, there is at least one sphere with self-intersection number $(-3)$ or $(-4)$, except one family which contains $5+q$ spheres, and one of them has self-intersection $(-6)$. The linear graphs which can be rationally blown include the examples of Fintushel and Stern and the more general examples of Park, where the continued fraction expansion of the weights is $-\frac{p^2}{pq-1}$ for $gcd(p,q)=1$. There is an recursive procedure to build all these linear plumbings described in \cite{SSW} section 4, which we review here. 
		
		Starting with a graph with one vertex labeled $(-4)$, one vertex labeled $(-1)$, and two edges between them we choose one of the two edges emanating from the $(-1)$ vertex and blow-up along that edge, meaning decrease the labels on each of the adjacent vertices by $1$, and insert a new vertex labeled $(-1)$ in the middle of the edge. Repeat this process, always blowing up along an edge emanating from the unique $(-1)$ vertex. We obtain linear graphs from these cyclic ones by deleting the unique $(-1)$ vertex along with its adjacent edges. See figure \ref{fig:linearQBD} for the first few steps. The linear graphs which can be obtained by this procedure form the entire family of Park's generalized linear rational blow-downs. Notice that if the final linear graph contains $N$ vertices, then it was obtained by this procedure by performing $N-1$ blow-ups starting with the base $-4$ case. Without loss of generality, we may suppose the first blow-up was done along the edge which emanates counter-clockwise from the $(-1)$ vertex. Suppose we perform a total of $m_1\geq 1$ blow-ups along the counter-clockwise emanating edge before performing one on the clockwise emanating edge. Then we obtain a vertex labeled $(-4-m_1)$ adjacent to the $(-1)$ vertex along with $m_1$ $(-2)$ vertices on the other side. Then suppose we perform $m_2$ blow-ups along the clockwise emanating edge before switching back. This has the effect of changing one of the $(-2)$ vertices to a $(-2-m_2)$ vertex and leaving a chain of $(-2)$ vertices on the other side. As we switch back and forth performing $m_j$ blow-ups on a given side each time, we produce a linear graph whose vertices have labels $(-4-m_1), (-2-m_2), \cdots , (-2-m_n)$ along with many $(-2)$ vertices. The total number of vertices at the end is $1+\sum_j m_j$.
		
		\begin{figure}
			\centering
			\includegraphics[scale=.5]{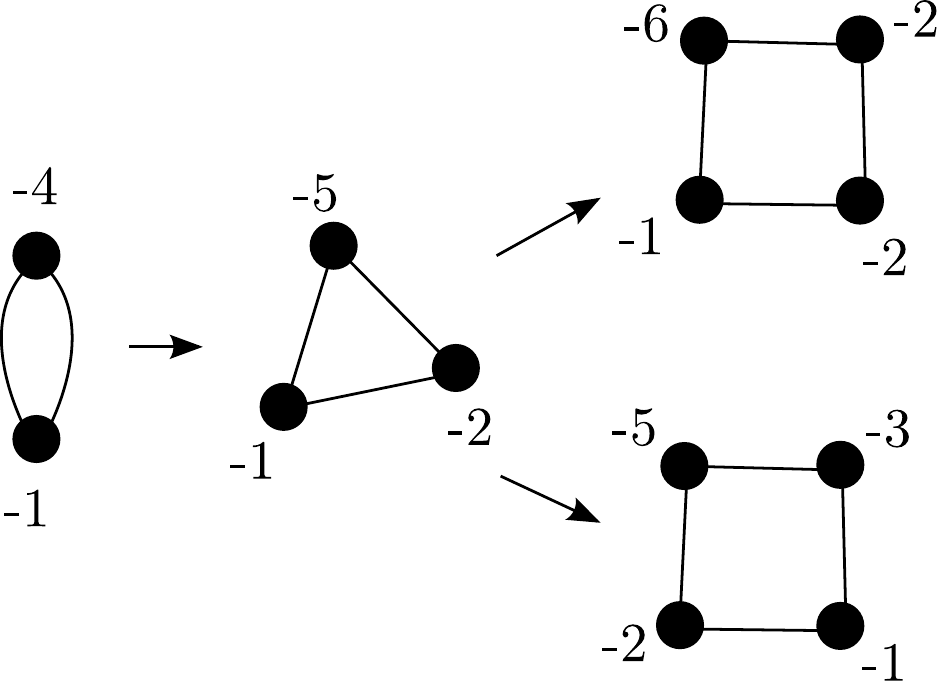}
			\caption{Recursive procedure to build linear graphs which can be rationally blown-down.}
			\label{fig:linearQBD}
		\end{figure}
	
		If the plumbing $P_{a,b}$ could be symplectically rationally blown down to obtain $L_{a,b}$, there must exist a plumbing with $ab$ vertices, which can be rationally blown-down and can embeds into $P_{a,b}$. We use the adjunction formula to rule out this possibility. $H_2(P_{a,b};\Z)$ is generated by the core spheres of the plumbing. Let $S_0$ denote the central sphere, and $S_i$ for $i\in \{1,\cdots, 2ab-a-b\}$, denote the spheres in the arms. 
		
		Each core sphere is symplectic so the adjunction formula holds: $\langle c_1(\omega), S_i\rangle = [S_i]^2+2$. Therefore, $\langle c_1(\omega),[S_0]\rangle =-a-b$ and $\langle c_1(\omega),[S_i]\rangle =0$ for $i\geq 1$. Now for any other symplectic sphere $S$ embedded in $P$, we can write $[S]=\sum_{i=0}^{ab-a-b}a_i[S_i]$. Then $[S]^2+2=\langle c_1(\omega),[S]\rangle = -(a+b)a_0$. In particular, $[S]^2+2$ must be divisible by $(a+b)$. By definition, we assume $a,b\geq 2$. If $a+b>4$ then any plumbing which contains a symplectic sphere $S$ of square $-3$, $-4$, or $-6$ cannot embed into $P_{a,b}$ because then $[S]^2+2\in \{-1,-2,-4\}$ which is never divisible by $a+b> 4$. Furthermore, when $a=b=2$ it is not possible to embed the plumbings with $-3$ or $-4$ spheres and the other non-linear examples have at least $5>ab$ vertices so those could not be used to get from $P_{2,2}$ to $L_{2,2}$. To rule out the linear rational blow-downs we use the description above of these graphs. If all of the spheres $S$ in the linear graph have the property that $[S]^2+2$ is divisible by $(a+b)$ then by the above discussion, $(a+b)$ divides $(-2-m_1), -m_2, \cdots, -m_n$. Therefore $(a+b)$ divides $2+\sum_j m_j$. On the other hand if the rational blow-down changes $P_{a,b}$ to $L_{a,b}$ then it must rationally blow-down a graph with $ab$ vertices so $ab=1+\sum_jm_j$. We conclude that $ab+1$ is divisible by $(a+b)$. Therefore under the hypotheses of the theorem, $L_{a,b}$ cannot be obtained from $P_{a,b}$ by any sequence of rational blow-downs.
		
	\end{proof}

	\begin{remark}
		We expect that all of the star surgery operations which replace $P_{a,b}$ by $L_{a,b}$ are inequivalent to sequences of rational blow-downs (as well as many other star surgeries). The hypothesis of the current proof ensures that we cannot embed at least one sphere of any appropriately sized rational blow-down. Lattice embedding arguments could be used to prove stronger results in additional cases.
	\end{remark}
	
	\begin{remark}
		Interpreting these star surgery operations as monodromy substitutions, we conclude that the relation in the planar mapping class monoid given by Lemma \ref{lem:MCGe0neg} cannot be generated by the lantern and daisy relations or the other relations corresponding to other rational blow-downs given in \cite{EMV}.
	\end{remark}

\bibliography{References}
\bibliographystyle{alpha}

\end{document}